\documentclass[11pt]{amsart}
\usepackage{graphicx}
\usepackage[active]{srcltx}
 \makeatletter
\renewcommand*\subjclass[2][2000]{%
  \def\@subjclass{#2}%
  \@ifundefined{subjclassname@#1}{%
    \ClassWarning{\@classname}{Unknown edition (#1) of Mathematics
      Subject Classification; using '1991'.}%
  }{%
    \@xp\let\@xp\subjclassname\csname subjclassname@#1\endcsname
  }%
}
 \makeatother
\usepackage{enumerate,url,amssymb,  mathrsfs,yhmath}

\newtheorem{theorem}{Theorem}[section]
\newtheorem{lemma}[theorem]{Lemma}
\newtheorem*{lemma*}{Lemma}
\newtheorem{proposition}[theorem]{Proposition}
\newtheorem{corollary}[theorem]{Corollary}

\def\1ton{1,2,\ldots,n}

\def\Div{{\rm Div}}

\def\d{\textnormal d}

\usepackage{amssymb}
\usepackage{amsthm}
\usepackage{mathrsfs, amsfonts, amsmath}
\usepackage{graphicx}
\usepackage{psfrag}
\newcommand{\norm}{\,|\!|\,}
\newcommand{\bydef}{\stackrel{{\rm def}}{=\!\!=}}

\newcommand{\onto}{\xrightarrow[]{{}_{\!\!\textnormal{onto}\!\!}}}

\newcommand{\Tr}{\text{Tr}}
\newcommand{\A}{\mathbb{A}}

\newcommand{\X}{\mathbb{X}}

\newcommand{\W}{\mathscr{W}}

\theoremstyle{definition}

\newtheorem{conjecture}[theorem]{Conjecture}

\theoremstyle{remark}
\newtheorem{remark}[theorem]{Remark}

\numberwithin{equation}{section}

\newcommand{\abs}[1]{\lvert#1\rvert}
\newcommand{\C}{\mathbb{C}}

\newcommand{\E}{\mathcal{E}}

 \DeclareMathOperator{\re}{Re}
\DeclareMathOperator{\im}{Im} \DeclareMathOperator{\id}{id}
\DeclareMathOperator{\Mod}{Mod}

\def\XXint#1#2#3{{\setbox0=\hbox{$#1{#2#3}{\int}$}
\vcenter{\hbox{$#2#3$}}\kern-.5\wd0}}

\def\ge{\geqslant}
\begin{document}


\title[Lipschitz property of minimisers between surfaces]{Lipschitz property of minimisers between double connected surfaces}  \subjclass{Primary 31A05;
Secondary 42B30 }


\author{David Kalaj}
\address{University of Montenegro, Faculty of Natural Sciences and
Mathematics, Cetinjski put b.b. 81000 Podgorica, Montenegro}
\email{davidk@ac.me}

\begin{abstract}
We study the  global Lipschitz character of minimisers of the Dirichlet energy of diffeomorphisms between  doubly connected domains with smooth boundaries from Riemann surfaces. The key point of the proof is the fact that minimisers are certain Noether harmonic maps,  with Hopf differential of special form, a
 fact invented by Iwaniec, Koh, Kovalev and Onninen in \cite{inv} for Euclidean metric and by the author in \cite{calculus} for the arbitrary metric, which
 depends deeply on a result of Jost \cite{Job1}. 
\end{abstract}
\maketitle
\tableofcontents

\section{Introduction and overview}\label{intsec}
Let $0<r<R$, $0<r_\ast<R_\ast$ and let  $\mathbb{X}$ and $\mathbb{Y}$ be two domains  in the complex plane $\mathbf{C}\cong\mathbf{R}^2$. Let $\rho$  be a continuous function on the closure of $\mathbb{Y}$.
The $\rho-$ Dirichlet energy integral of a mapping $h\in \W^{1,2}(\mathbb{X}, \mathbb{Y})$  is defined by \begin{equation}\label{penergy}\mathscr{E}^\rho[h]=\int_{X} \rho(h(z)) \|Dh(z)\|^2 dz.\end{equation} The central aim of this paper is to get some boundary regularity of the  minimizer of  the $\rho-$ energy integral of homomorphisms from the Sobolev class $\W^{1,2}(\mathbb{X}, \mathbb{Y})$.

The main result of this paper is 

\begin{theorem} \label{mainexistq}
Suppose  that $D$ and $\Omega$ are double  connected domains in
$\C$ with $C^2$ boundaries and let $\rho\in C^2(\Omega)$ be a real nonvanishing function in the closure of $\Omega$. Then every energy minimising  diffeomorphism  of
$\rho-$ energy between $D$ and $ \Omega$,
is  Lipschitz  continuous up to the boundary of  $D$. However it is not bi-Lipschitz in general.
\end{theorem}

The paper is consisted of this section and  three more sections.

In the following  subsections, we present three different type of harmonic mappings.  Further in the section~2 we make  some background and reformulate main result in the therm of harmonic mappings. In section~3 we define the class of $(K,K')-$quasiconformal mappings and prove that stationary points of the energy take part on this class. In the section~4 we prove the main result. In the last subsection are performed some precise calculations of Lipschitz constants for minimisers of energy for radial metrics and circular annuli.

\subsection{Harmonic mappings}
Assume that $\X$ is domain in $\mathbf{R}^2$  (for example $\X$ is homeomorphic  to an circular annulus $\{x\in \mathbf{R}^2 | 1<|x|<R\}$).
The classical Dirichlet problem concerns the energy minimal mapping $h \colon \mathbb{\X} \to \mathbb{R}^2$ of the Sobolev class $h\in h_\circ +
\W^{1,2}_\circ (\mathbb{\X}, \mathbb{R}^2)$ whose boundary values are explicitly prescribed by means of a given mapping $h_\circ \in  \W^{1,2}
(\mathbb{A}, \mathbb{R}^2)$.  Let us consider the variation   $h \leadsto h\,+ \,\epsilon \eta $,  in which $\eta \in \mathscr C^\infty_\circ (\mathbb{X} ,
 \mathbf{R}^2)$ and $\epsilon \to 0$, leads to the integral form of the familiar harmonic system of equations
\begin{equation}\label{equa1}
\int_{\mathbb{X}} \left(\left<\nabla \rho, \eta\right>\norm Dh\norm ^2+\langle \rho(h)  Dh , \, D\eta \rangle \right) =0, \ \mbox{ for every } \eta
 \in \mathscr C^\infty_\circ (\mathbb{X} , \mathbf{R}^2).
\end{equation}
Equivalently
\begin{equation}\label{equa2}
\Delta^\rho h \bydef \left(\Div\big( \rho(h)Dh_1\big),\Div\big( \rho(h)Dh_1\big)\right)-\frac{1}{2}\norm Dh\norm ^2\nabla \rho=0,
\end{equation}
in the  sense of distributions. Here $h=(h_1,h_2)$. Then by using the complex notation \eqref{equa2} can be written as
\begin{equation}\label{eqvi}\tau(h)\equiv h_{z\overline z}+{(\log \rho)}_w\circ h\cdot  h_z\,h_{\bar z}=0.
\end{equation}
The solutions to the equation \eqref{eqvi} are called \emph{weak harmonic mappings} or simply harmonic mapping (see the Remark~\ref{popa}, (1) for the explanation).

On the next subsection  we derive the general harmonic equation which by using a different variation as the following.
\subsection{General harmonic mappings (cf. \cite{helein2})} The situation is different if  we allow $h$ to slip freely along the boundaries. The {\it inner variation} come to stage in this case. This is simply a change of the  variable; $h_\epsilon=h \circ \eta_\epsilon $, where  $\eta_\epsilon \colon \X \onto \X$ is a $\mathscr C^\infty$-smooth diffeomorphsm  of $\X$ onto itself, depending smoothly on a parameter $\epsilon \approx 0$ where  $\eta_\circ = id \colon \X \onto \X$.

 Let us take on the inner variation of the form
\begin{equation}\label{equa7}
\eta_\epsilon (z)= z + \epsilon \, \eta (z), \qquad \eta \in \mathscr C_\circ^\infty (\X, \mathbf{R}^2).
\end{equation}
By using the notation  $w=z+\epsilon \, \eta (z) \in \X$,
we obtain
$$\rho(h_\epsilon)Dh_\epsilon (z) = \rho(h(w)) Dh (y) (I+ \epsilon D\eta(z)).$$ Hence
\[
\begin{split}
\rho(h_\epsilon(z))\norm Dh_\epsilon(z)\norm^2 & = \rho(h(w))\norm Dh(y)\norm^2
\\&+ 2 \epsilon\,  \rho(h(w))\langle  D^\ast h(w)\cdot  Dh(w)\, ,\,  D \eta \rangle + o(\epsilon).
\end{split}
\]
By integrating with respect to $x\in \X$ we obtain
\[\begin{split}\mathscr E_\rho[h_\epsilon] &=\int_\X \rho(h_\epsilon(z))\norm Dh_\epsilon(z)\norm^2 dz\\&= \int_\X \bigg[
\rho(h(w))\norm Dh(w) \norm^2 \\& \ \ \ \ +  2 \epsilon \rho(h(w))\langle  D^\ast h(w)\cdot  Dh(w)\, ,\,
  D \eta(z) \rangle \bigg]\, \d z + o(\epsilon).\end{split} \]
We now make the substitution $w=z + \epsilon \, \eta (z)$, which is a diffeomorphism for small $\epsilon$, for which we have:
$z= w- \epsilon \, \eta (w)+ o(\epsilon)$, $D\eta (z)= D\eta (w)+o(1)$, when $\epsilon\to 0$, and the change of volume element
$\d z = [1-\epsilon \, \Tr \,D \eta (w) ]\, \d w + o(\epsilon) $. Further
$$\int_\X \rho(h(w))\norm Dh(w) \norm^2 \d z=\int_\X \rho(h(w))\norm Dh(w) \norm^2 [1-\epsilon \, \Tr \,D \eta (w) ]\, \d w + o(\epsilon).$$
 The so called equilibrium equation for the inner variation is obtained from \begin{equation}\label{ddep}\frac{\d}{\d \epsilon}\bigg|_{\epsilon=0} \mathscr E_{h_\epsilon}\,=\,0\,\end{equation}
\begin{equation}\label{intstar}\int_\X \langle \rho(h) D^\ast h \cdot Dh - \frac{\rho(h)}{2} \norm Dh \norm^2 I \, , \, D \eta \rangle \, \d w=0 \end{equation}
or, by using distributions
\begin{equation}\label{enhe}
\Div \left(\rho(h) D^\ast h \cdot Dh - \frac{\rho(h)}{2} \norm Dh \norm^2 I  \right)=0.
\end{equation}

This equation \eqref{enhe} is known as the Hopf equation, and the corresponding differential is called the Hopf differential. Since for $h(z)=(a(z),b(z))$, we have $$\rho(h)D^\ast h \, Dh - \frac{\rho(h)}{2} \norm Dh \norm^2 I  =\left(
                                      \begin{array}{cc}
                                        U & V \\
                                        V & -U \\
                                      \end{array}\right),$$ where $$U = \frac{\rho(h)}{2} (a_x^2+b_x^2-a_y^2-b_y^2)$$ and $$V=\rho(h)(a_xa_y+b_xb_y),$$                                 then \eqref{enhe} in complex notation  takes the form $$(U_x +U_y)- i (V_x+V_y)=0$$ or what is the same
\begin{equation}\label{eqv}
\frac{\partial}{\partial \bar z} \left(\rho(h(z))h_z \overline{h_{\bar z}}\right)=0, \qquad z= x+iy.
\end{equation}
The solution to \eqref{eqv} is called the \emph{general $\rho-$ harmonic mapping}. Assume that $h\in \mathscr{C}^2$ and assume that $h$ satisfies \eqref{eqvi} . Then by direct calculation we obtain
\[\frac{\partial}{\partial \bar z} \left(\rho(h(z))h_z \overline{h_{\bar z}}\right)=\rho(h(z))\left(\bar h_z \cdot \tau(h)+h_z\cdot \overline{\tau(h)}\right)=0.\]
This implies that every harmonic mapping is general harmonic mapping.
\subsection{Noether harmonic mappings (cf. \cite{helein2})} 
We call a mapping  $h$ Noether harmonic  if
\begin{equation}\label{stat}
\frac{\mathrm{d}}{\mathrm{dt}}\bigg|_{t=0}{\mathcal{E}^\rho}[h\circ \phi_t^{-1}]=0
\end{equation}
for every family of diffeomorphisms $t\to \phi_t\colon
\Omega\to\Omega$ which depend smoothly on the parameter $t\in\mathbb
R$ and satisfy
$\phi_0=\id$. The latter mean that the mapping $\Omega\times [0,\epsilon_0]\ni (t,z)\to \phi_t(z)\in \Omega $ is a smooth mapping for some $\epsilon_0>0$. It is clear by the definition that every Noether harmonic mapping is general harmonic mapping, and therefore its Hopf differential is holomorphic. Namely the equation \eqref{stat} implies the equation \eqref{ddep}.

In the following remark we summarize the difference between harmonic mappings, general harmonic mappings and Noether harmonic mappings.
\begin{remark}\label{popa} Assume that $h$ is a mappings between two domains of the complex plane $\mathbf{C}$.
\begin{itemize}
                     \item[(1)] Every weak solution to \eqref{eqvi} which belongs to $\mathscr{W}^{1,2}$ is smooth (see paper of H$\mathrm{\acute{e}}$lein \cite{helein} see also the Remark after \cite[Definition~1.3.1]{Job1}), and thus it is a strong solution of \eqref{eqv}. Moreover it satisfies the equation \eqref{eqv}, i.e. it is a general $\rho-$harmonic mapping.
                     \item[(2)] There are general harmonic mappings  that are not weakly harmonic mappings. If $h\in C^2$ or $h\in C^1$ and $J(z,h)\neq 0$ (see \cite{jost85}) then a general harmonic mapping is a harmonic mapping.
                     \item[(3)] There are general $\rho-$harmonic mappings that are not Noether harmonic mappings. Namely the Hopf differential of Noether harmonic mappings are very special. (See subsection~\ref{sube} for details).
                   \end{itemize}
\end{remark}
\subsection{Some key properties of Noether harmonic diffeomorphisms}\label{sube}
 Two of following key properties of the Noether harmonic mappings are derived in \cite{calculus}:
\begin{itemize}
\item[1.] The function $\varphi:= \rho(g(z)) g_z\overline{g_{\bar z}}$, a priori in $L^1(D)$, is holomorphic.
\item[2.] If $\partial D$ is $\mathscr{C}^1$-smooth then $\varphi$ extends continuously to $\overline{D}$, and the quadratic differential $\varphi \, dz^2$  is real on each boundary curve of $D$.
\end{itemize}

Further by using those key properties in \cite{calculus} it is shown the following statement. Let  $D=A(r,R)$ be a circular annulus, $0<r<R<\infty$, and
$\Omega$ a doubly connected domain. If $g$ is a stationary diffeomorphism, then
\begin{equation}\label{hopf1}\rho(g(z))
g_z\overline{g_{\bar z}} \equiv \frac{c}{z^2}\qquad \text{in }D
\end{equation}
where $c\in\mathbf{R}$ is a constant.

Throughout this paper $M=(D,\sigma)$ and $N=(\Omega,\rho)$ will be
doubly connected domains in the complex plane $\C$, where $\rho$ is a
non-vanishing smooth metric defined in $\Omega$ so that:

\begin{enumerate}
  \item It has a bounded Gauss
curvature $\mathcal{K}$ where
\begin{equation*}\label{gaus}\mathcal{K}(w)=-\frac{\Delta \log \rho(w)}{\rho(w)};\end{equation*}
  \item It has a finite area defined by
$$\mathcal{A}(\rho)=\int_{\Omega}\rho(w) du dv, \ \ w=u+iv;$$
  \item There is a constant $P>0$ so that \begin{equation}\label{pp}{|\nabla \rho(w)|}\le P{\rho(w)}, \ \ \ w\in\Omega,\end{equation} which means that $\rho$ is so-called approximately analytic function (c.f. \cite{EH}).
\end{enumerate}
 We call such a metric
\emph{admissible} one. The Euclidean metric is an admissible metric. The Riemanian metric defined by $\rho(w)=\frac{1}{(1+|w|^2)^2}$ is admissible as well.
The Hyperbolic metric $\lambda(w)=\frac{1}{(1-|w|^2)^2}$ is not an admissible metric on the unit disk neither on the annuli $\A(r,1)\bydef \{z:r<|z|<1\}$,
but it is admissible in $\A(r,R)\bydef\{z: r<|z|<R\}$, where $0<r<R<1$.  In this case the equation \eqref{eqvi} leads to hyperbolic harmonic mappings. The class is particularly interesting, due to recent discovery that every quasisimmetric map of the unit circle onto itself can be extended to a quasiconformal hyperbolic harmonic mapping of the unit disk onto itself. This problem  is known as the Schoen conjecture and it was proved by Markovi\'c in \cite{markovic}.

\section{Some background and precise statement of the results}

The primary goal of this paper is to study some Lipschitz behaviors of the   minimisers of the functional
$\E^\rho[g]$.  We will study the
Lipschitz continuity of the diffeomorphisms $f\colon D\onto\Omega$ of
smallest $\rho-$Dirichlet energy where $\rho$ is an
arbitrary smooth metric with bounded Gauss curvature and finite
area.
Notice first that a change of variables $w=f(z)$ in~\eqref{penergy}
yields
\begin{equation}\label{ener2}
{\E^\rho}[f] = 2\int_{D} \rho(f(z))J(z,f)\, dz +
4\int_{D}\rho(f(z)) \abs{f_{\bar z}}^2\ge 2 \mathcal{A}(\rho)
\end{equation}
where $J(z,f)$ is the Jacobian determinant of $f$ at $z$ and $\mathcal{A}(\rho)$ is
the area of $\Omega$. A conformal mapping
$f:D\onto\Omega$, which exists due to the celebrated Riemann mapping theorem; that is, a homeomorphic solution of the
Cauchy-Riemann system $ f_{\bar z}=0$, is an  obvious
 minimiser of~\eqref{ener2}. The boundary behaviors  of conformal mappings between planar domains are well-established. We refer to the book of Pommerekne \cite{pom}. Two results that are of broad interest are
\begin{enumerate}
\item the Charath\'eodory theorem, which states that every conformal mapping between two Jordan domains has a continuous extension to the boundary

\item the results of Warshawski's and Kellogg that every conformal mapping between $C^{k,\alpha}$ Jordan domains has $C^{k,\alpha}$ extension to the boundary. Here $k$ is a positive integer and $\alpha\in(0,1)$.
\end{enumerate}
In particular we have
\begin{proposition}
If $f$ is a conformal mapping between two Jordan domains with smooth boundary, then $f$ is Lipschitz continuous.
\end{proposition}

The doubly connected case, being next in the order of
complexity, is the subject of the further results. Conformal mappings are not minimisers for arbitrary doubly  connected
domains provided that the domains are not conformally equivalent.

The case of circular
annuli w.r.t. Euclidean metric and the metric $\rho(w)=1/|w|$ is
fully established in \cite{AIM} by Astala, Iwaniec and Martin, where it is shown that the radial harmonic mappings are minimisers.  This result has been extended to
all radial metrics in \cite{klondon} by Kalaj. The regularity of the class of radial mappings is a simple issue since they have explicit expression.

Concerning the existence,   Koh, Kovalev, Iwaniec and Onninen in \cite{inv} proved that there exists a harmonic diffeomorphism which minimizes the Euclidean energy in the class of Sobolev homeomorphisms between doubly  connected domains in the complex plane, provided that the domain has smaller modulus than
the target. Then this result has been extended for arbitrary  metric with bounded area and Gaussian curvature by the author in \cite{calculus}, where it is proved the following theorem.
\begin{proposition} \label{mainexist2}
Suppose that $D$ and $\Omega$ are doubly connected domains in
$\C$ such that $\Mod D\le \Mod\Omega$ and let $\rho\in C^2(\Omega)$ be a metric in $\Omega$ with Gaussian curvature bounded from above and assume that the metric has area  $\mathcal{A}(\rho)<\infty.$ Then there exists an
$\rho-$energy-minimal diffeomorphism $f: D\onto \Omega$,
which is  $\rho-$Noether harmonic (and consequently a $\rho-$ harmonic) and is unique up to a conformal change of variables in $D$.
\end{proposition}
Concerning some behaviors that minimisers of Euclidean energy inherit inside of the double connected domain, provided that the image domain is bounded by convex curves
or by two circles  we refer to the recent papers by Koh \cite{koh1} and \cite{koh2}.
Now we reformulate the main result of this paper in which we establish the boundary behaviors of minimisers.
\begin{theorem} \label{mainexist}
Suppose that $M=(D,\sigma)$ and $N=(\Omega,\rho)$ are Riemannian surfaces, so that $D$ and $\Omega$ are double  connected domains in
$\C$ with $C^2$ boundaries and let $\rho\in C^2(\Omega)$ be an admissible metric. Then every Noether harmonic diffeomorphism  of
$\rho-$ between $D$ and $ \Omega$,
is  Lipschitz  continuous up to the boundary of  $D$. However it is not bi-Lipschitz in general.
\end{theorem}

Proposition~\ref{mainexist2} and Theorem~\ref{mainexist} imply the following
\begin{corollary} \label{mainexist1}
Suppose that $D$ and $\Omega$ are doubly connected domains in
$\C$ with  $C^2$ smooth boundaries such that $\Mod D\le \Mod\Omega$ and let $\rho\in C^2(\Omega)$ be a metric in $\Omega$ with finite area and  Gaussian curvature bounded from above. Then there is a
$\rho-$energy-minimal Noether harmonic diffeomorphism  $f: D\onto \Omega$ which is Lipschitz  continuous up to the boundary of  $D$.
\end{corollary}


\section{$(K,K')-$quasiconformal mappings }\label{stasec}

A sense preserving mapping $w$ of class ACL between two planar domains $D$ and $D$ is called $(K, K')$-quasi-conformal if \begin{equation}\label{map}\|Dw\|^2\le 2KJ(z,w)+K',\end{equation} for almost every $z\in D$. Here $K\ge 1, K'\ge 0$, $J(z,w)$ is the Jacobian of $w$ in $z$ and $\|Dw\|^2=|w_x|^2+|w_y^2|=2|w_z|^2+2|w_{\bar z}|^2$.
Since $$|Dw|=|w_z|+|w_{\bar z}|,$$ from \eqref{map} it follows that

\begin{equation}\label{mapa}|Dw|^2\le 2KJ(z,w)+K'.\end{equation}

Mappings which satisfy Eq. \eqref{map} arise naturally in elliptic equations, where $w =
u + iv$, and $u$ and $v$ are partial derivatives of solutions (see  \cite[Chapter~XII]{gt} and the paper of Simon \cite{simon}).

\subsection{Noether harmonic maps and $(K,K')-$ quasiconformal mappings}
Now we want to prove the following important property of Noether harmonic maps
\begin{lemma}\label{popi}
Every Noether harmonic map $g:\A(r,1)\to \Omega$ is $(K,K')$ quasiconformal, where $$K=1 \ \ \text{and}\ \ K'=\frac{2\abs{c}}{r^2\inf_{w\in \Omega}\rho(w)}.$$ The result is sharp and for $c=0$ the Noether harmonic map is $(1,0)$ quasiconformal, i.e. it is a conformal mapping. In this case $\Omega$ is conformally equivalent with $\A(r,1)$.
\end{lemma}
\begin{proof}[Proof of Lemma~\ref{popi}]

Let $N=\frac{z}{|z|}$ and $T=i N$. Then we define $$g_N(z)=Dg(z) N=\frac{z}{|z|}g_z+\frac{\overline{z}}{|z|}g_{\bar z} $$ and $$g_T(z)=Dg(z) T=\frac{zi}{|z|}g_z+\frac{\overline{zi}}{|z|}g_{\bar z} .$$ Then it is clear that $$\|Dg(z)\|^2=|g_N|^2+|g_T|^2.$$ Further $$|g_N|^2-|g_T|^2=4 \mathrm{Re}\left(\frac{z^2}{|z|^2}g_z\overline{g_{\bar z}}\right).$$
By using now \eqref{hopf1} we arrive at the equation
\begin{equation}
\label{hopf2a}\rho(g(z))(\abs{g_N}^2 - \abs{g_T}^2) = \frac{4c}{\abs{z}^2}. \end{equation} In a similar way we get
\begin{equation}\label{hopf2b}\begin{split}\rho(g(z))\re (\overline{g_N} g_T) &=\rho(g(z))\re \left[{\left(\frac{\bar z}{|z|}\overline{g_z}+\frac{z}{|z|}\overline{g_{\bar z}}\right)}\cdot \left(\frac{zi}{|z|}g_z+\frac{\overline{zi}}{|z|}g_{\bar z}\right)\right]\\&=\rho(g(z))\mathrm{Im}\left(\frac{z^2}{|z|^2}g_z\overline{g_{\bar z}}\right)=0.\end{split} \end{equation}
Further we have that $$J(z,g) =|g_{ z}|^2-|g_{\bar z}|^2= \im (\overline{g_N} g_T)\ge 0,$$ which in view
of~\eqref{hopf2b} reads as
\begin{equation}\label{hopf5}
J(z,g) = \abs{g_N} \abs{g_T}.
\end{equation}
Now $$\abs{g_N}^2 - \abs{g_T}^2=\frac{4c}{\rho(g(z))\abs{z}^2}.$$
So \begin{equation}\begin{split}\|Dg\|^2-2J(z,g)&=\abs{g_N}^2 +\abs{g_T}^2-2 \abs{g_N} \abs{g_T}\\&=(\abs{g_N} - \abs{g_T})^2\\&\le(\abs{\abs{g_N} - \abs{g_T}})(\abs{g_N} + \abs{g_T})\\&=\abs{\abs{g_N}^2 - \abs{g_T}^2}= \frac{4\abs{c}}{\rho(g(z))\abs{z}^2}.\end{split}\end{equation}
This implies the claim.

\end{proof}

\subsection{Distance function and $(K,K')-$quasiconformal mappings}
Let $\Omega$ be double connected domain with boundary $\partial \Omega\in C^2$. Then $\Omega=\Omega_1\setminus \Omega_2$ for two bounded Jordan domains with $C^2$ boundaries $\partial \Omega_1$ and $\partial \Omega_2$. The
conditions on $\Omega$ imply that $\partial \Omega$ satisfies the
following condition: at each point $w\in \partial \Omega$ there
exists a disk $\Omega = D(w_{w}, r_z)$ depending on $z$ such that
$\overline \Omega\cap (\Bbb C\setminus \Omega) = \{w\}$. Moreover $\mu :
= \inf\{r_w, w\in \partial \Omega\}>0$.

It is easy to show that
$\mu^{-1}$ bounds the curvature of $\partial \Omega$, which means
that $\frac{1}{\mu} \ge {\kappa_z},$ for $z\in \partial \Omega$.
Let $d_1$ be the distance function with respect to the boundary of the
domain $\Omega_1$: $d_1(w)=\mathrm{dist}(w,\partial \Omega_1)$. Let
$\Gamma_\mu:=\{w\in \Omega :d_1 (w)\le \mu\}$. For basic properties of distance
function we refer to \cite{gt}. For example $\nabla d_1(w)$ is a unit
vector for $w\in \Gamma_\mu$, and $d_1\in C^2(\overline
{\Gamma_\mu})$ because $\partial \Omega\in C^2$.

Under the above conditions for $w\in \Gamma_\mu$ there exists $\zeta_1(w)\in \partial \Omega_1$ such that
\begin{equation}\label{distnorm} \nabla d_1(w)=\mathbf{\nu}(\zeta(w)),
\end{equation} where $\mathbf{\nu}(\zeta(w))$ denotes the inner unit normal vector at $\zeta(w)\in\partial \Omega$.
See \cite{gt} for details.
We now have.
\begin{lemma}\label{helpa}
Let $w:D\mapsto \Omega$ be a $(K,K')-$ quasiconformal mapping and
$\chi=-d_1(w(z))$. Let $\kappa_0=\mathrm{ess}\
\sup\{|\kappa_z|:z\in\partial \Omega\}$ and  $0<\mu<{\kappa^{-1}_0}$. Then:
\begin{equation}\label{dqh} |\nabla\chi(z)|\le|D w(z)|\le 2K|\nabla
\chi(z)|+\sqrt{K'}
\end{equation}
for $z\in w^{-1}(\Gamma_\mu)$.
\end{lemma}
\begin{proof}
Observe first that $\nabla d_1$ is a unit vector. From $\nabla
\chi=-\nabla  d_1\cdot D w$ it follows that
$$|\nabla\chi|\le|\nabla d_1||Dw|=|D w|.$$

 Since $w$ is
$(K,K')$-q.c., it follows from \eqref{mapa} the inequality
$$\left|D w\right|^2\le 2K J(w,z)+K'= 2K |D  w| l(D  w)+K'.$$ Then we have
$$\left|D w\right|\le 2K l(D  w)+\sqrt{K'}$$ Next we have
that $(\nabla \chi)^T=-(D w)^T\cdot (\nabla d_1)^T$ and therefore
for $z\in w^{-1}(\Gamma_\mu)$, we obtain $$|\nabla \chi|\ge
\inf_{|e|=1}|(D w)^T \,e|=\inf_{|e|=1}|D w \,e|=l(w)\ge
\frac{|D w|}{2K}-\frac{\sqrt{K'}}{2K}.$$
The proof of (\ref{dqh}) is completed.
\end{proof}

\section{Proof of the main result }

\begin{proof}[Proof of  Theorem~\ref{mainexist}]

First of all since $f$ is a diffeomorphism, according to Remark~\ref{popa} ) $f$ satisfies the harmonic mapping equation
\begin{equation}\label{heli}f_{z\bar z}+\frac{\partial \log \rho(w)}{\partial w} \circ f(z) \cdot f_{z}\cdot f_{\bar z}=0.\end{equation}
Now we define $\chi(z)=-d_1(w(z))=-\mathrm{dist}(w(z),\partial \Omega _1)$. By repeating the proof of the corresponding result in \cite{pisa} we get the following
\begin{lemma}
Let $w:\A(r,1)\mapsto \Omega$ be a twice differentiable mapping and
let $\chi(z)=-d_1(w(z))=-\mathrm{dist}(w(z),\partial \Omega _1)$, where $\partial \Omega _1$ is the outer boundary of $\Omega$. Then
\begin{equation}\label{lapdistance}\Delta \chi(z) =
\frac{\kappa_{w_\circ}\cdot |(O_{z} D
w(z))^te_1|^2}{1-\kappa_{w_\circ} d_1(w(z))} -\left<(\nabla d_1)(w(z)), \Delta
w\right>,\end{equation} where $e_1=(1,0)$, $z \in w^{-1}(\Gamma_\mu)$,
$w_\circ\in\partial \Omega_1$ with $|w(z)-w_\circ|=\mathrm{dist}(w(z),
\partial \Omega_1)$,  $\mu>0$
such that $1/\mu>\kappa_0=
\mathrm{ess}\sup\{|\kappa_w|:w\in\partial \Omega_1\}$ and $O_{z}$ is
an orthogonal transformation.
\end{lemma}
From \eqref{lapdistance}, \eqref{distnorm}, \eqref{heli}, \eqref{dqh} and the condition \eqref{pp} for the metric $\rho$, we have

\[\begin{split}|\Delta \chi(z)| &\le
\frac{\kappa_{0}}{1-\kappa_{0} \mu }{\cdot |(O_{z} D
w(z))^te_1|^2} +|\left<(\nabla d_1)(w(z)), \Delta
w\right>|\\&\le
\frac{\kappa_{0}}{1-\kappa_{0} \mu }{\cdot |Dw(z))|^2} +|\Delta
w|
\\&\le
\frac{\kappa_{0}}{1-\kappa_{0} \mu }{\cdot |Dw(z))|^2} +2P |Dw(z))|^2\\&=(\frac{\kappa_{0}}{1-\kappa_{0} \mu }+2P){\cdot |Dw(z))|^2}\\
&\le (\frac{\kappa_{0}}{1-\kappa_{0} \mu }+2P)\cdot \left(2K|\nabla
\chi|+\sqrt{K'}\right)^2\\
&\le a_1 |\nabla
\chi|^2+b_1, \end{split}
\]
  where $$a_1=4K^2\left(\frac{\kappa_{0}}{1-\kappa_{0} \mu }+2P\right)$$ and $$b_1=2K'\left(\frac{\kappa_{0}}{1-\kappa_{0} \mu }+2P\right).$$
On the other hand, because $w$ is a diffeomorphism  between $\A(r,1)$ and $\Omega$, it follows that  $\lim_{|z|\to 1}\chi(z)=0.$ Thus we can extend $\chi$ to be zero in  $|z|=1$. Let $\tilde\chi:\mathbf{U}\to \mathbf{R}$ be a $C^2$ extension
of the function $\chi|_{ w^{-1}(\Gamma_{\mu/2})}$. It exists in view of Whitney's theorem. Let
$b_0=\max\{|\Delta \tilde\chi(x)|:x\in
\mathbf{U}\setminus w^{-1}(\Gamma_{\mu/2})\}$. Then $$|\Delta \tilde \chi|
\le
a_1|\nabla \tilde\chi|^2+b_1+b_0.$$ Thus the conditions of the following
Lemma~\ref{heb} are satisfied.

\begin{lemma}[Heinz-Berenstein]\label{heb}\cite[Theorem~$4^\prime$]{EH}. Let $\chi: \overline{\mathbf{R}} \mapsto \mathbf{R}$
be a continuous function between the unit disc
$\overline{\mathbf{U}}$ and the real line satisfying the conditions:
\begin{enumerate}\item $\chi$ is $C^2$ on ${\mathbf{U}}$,
\item $\chi_b(\theta)= \chi(e^{i\theta})$ is $C^2$ and
\item $ |\Delta \chi| \leq a |\bigtriangledown \chi|^2+b $ on $\mathbf{U}$ for some constant $c_0$.
\end{enumerate} Then the gradient $|\bigtriangledown \chi|$ is bounded on
$\mathbf{U}$
\end{lemma}
The conclusion is that
$\nabla \tilde\chi$ is bounded.
Now Lemma~\ref{helpa} implies that there is a constant $C>0$ so that
\begin{equation}\label{11}|D w|\le C, \ \  \ z\in w^{-1}(\Gamma_{\mu/2}).\end{equation}
In order to deal with the inner boundary of $\Omega$ assume without loss of generality that $0\in \Omega_2$.
Now if $W(z)=1/\overline{w(r/\bar z)}$, then after straightforward calculation
\begin{equation}\label{Ww}\Delta W(z)=r^2\frac {(2 (\overline {w} _x)^2 +
    2 (\overline {w} _y)^2 - \overline {w}\overline {\Delta w})} {|
   z | ^2 \bar w^3} =r^2\frac {(8\overline{ w_{z}\cdot w_{\bar z}} - \overline {w}\overline {\Delta w})} {|
   z | ^2 \bar w^3}.\end{equation}
      Further  \begin{equation}\label{ww}\|DW\|=\frac{r \|Dw\|}{|z|^2|w|^2}\end{equation} and \begin{equation}\label{jj}J(z,W)=\frac{r^2 J(1/\bar z,w) }{|z|^4|w|^4},\end{equation} we get that $$|\Delta W(z)|\le a_2 \|DW\|^2 +b_2,  \ \ z\in \Omega'.$$ Further $W$ is $(K_1, K_1')$ quasiconformal with

   $$\|DW\|^2=\frac{r^2 \|Dw\|^2}{|z|^4|w|^4}\le \frac{2r^2 K\cdot J(1/\bar z,w)+r^2 K'}{|z|^4 |w|^4}\le K_1 J(z,W)+K_1'.$$

   Moreover $W$ maps $\A(r,1)$ onto $\Omega'=\Omega'_1\setminus \Omega'_2$.  By proceeding as in the first part we get that the mapping $\xi(z)=-d_1(W(z))=-\mathrm{dist}(W(z),\partial \Omega' _1)$ is Lipschitz near $\mathbf{T}\subset \partial \A(1,r).$
   Then again in view of Lemma~\ref{helpa} we conclude that $|DW|$ is bounded in $W^{-1}(\Gamma'_{\mu/2})$, where $\Gamma'_\sigma=\{z\in \Omega': \mathrm{dist}(z,\partial\Omega'_1)<\sigma\}$.
   Thus by \eqref{ww} there exists $\epsilon>0$ and $C_1>0$ so that \begin{equation}\label{22}|Dw(z)|\le C_1 , \
    \ \ r<|z|<r+\epsilon.\end{equation}

   Since $w$ is smooth in $\A(r,1)$ in view of \eqref{11} and \eqref{22}  we conclude that $w$ has a Lipschitz extension to $\overline{\A(r,1)}$.

In order to deal with the arbitrary  domain $D$ with $C^2$ boundary, we make use of the following Kellogg type result that follows from \cite[Theorem~3.1]{jost}.
\begin{proposition}
Suppose that $D$ is a double connected bounded by two Jordan curves of class $C^{1,\alpha}$ and assume that $r=\exp(-\mathrm{Mod}(D))$.
Then there exists a conformal diffeomorphism $\tau: D\to \A(r,1)$ which is $C^{1,\alpha}$ up to the boundary together with its inverse.
In particular $\tau$ is bi-Lipschitz.

\end{proposition}

Assume now that  $w: D\onto \Omega$ is a harmonic diffeomorphism that minimizes the the $\rho-$energy, where $D$ is not a circular annuli.
Then there exists a conformal mapping $\tau$ of $\A(r,1)$ onto $D$ which is bi-Lipschitz continuous.
Here $r=\exp(-\mathrm{Mod}(D))$. Then the mapping $\zeta(z)=w(\tau(z)):\A(1,r)\onto\Omega$ is a minimizer
that minimizes the $\rho-$energy and by the first part of the proof it has Lipschitz continuous extension up to the boundary.
 Now we conclude that $w$ is Lipschitz continuous up to the boundary and this finishes the proof.
\end{proof}
\begin{remark}
Let $f(z)=\int_0^z \frac{dw}{\sqrt{1-w^4}}$ be a conformal mapping of the unit disk onto a square.
Then $f$ is a conformal mapping of the annulus $\A(1/2,1)$ onto the doubly connected,
whose outer boundary is not smooth. We know that $f$ is a minimiser of energy but is not Lipschitz.
 With some more effort, by using e.g. \cite{les} we can define a conformal mapping between the circular annulus and an annulus with $C^1$
 boundary so that it is not Lipschitz up to the boundary. This in turn implies that the condition for the annuli to have $C^2$ boundary is essential.
  It seems that we can weaken the hypothesis on smoothness of the boundary, but we didn't make a serious effort in this direction.
   Further an Euclidean harmonic diffeomorphism $f$ of the unit disk $\mathbf{D}$ onto itself is seldom  a
    Lipschitz continuous up to the boundary. We cite here an important result of Pavlovi\'c \cite{MP} which states that harmonic diffeomorphism of the unit disk is
    Lipschitz if it is quasiconformal. Further for such a non-Lipschitz  $f$, let $R<1$.
    Then  the set $D=f^{-1}(\A(R,1))$ is a doubly-connected surface with $C^{\infty}$ boundary. Let $\varphi$ be a conformal mapping of the annulus $\A(r,1)$
     onto $D$. Then $F=f\circ\varphi$ is a harmonic diffeomorphism between $\A(r,1)$ onto $\A(R,1)$ which is not Lipschitz continuous. This observation tells
      us that there exists a  crucial difference between the Noether harmonic diffeomorphisms  and those harmonic diffeomorphisms between annuli which are not Noether harmonic.
\end{remark}

In the next subsection we get precise estimate for the case of radial metric and circular annuli and finish the last part of main theorem.
\subsection{Lipschitz continuity of minimisers for circular annuli}\label{sec2} 
Assume that $\rho:[r,1]\to (0,+\infty)$ is a smooth mapping with $\rho(s)\ge 1/M>0$. Then it defines the radial metric also denoted by $\rho$ in $\A(r,1)$, $\rho(z)=\rho(|z|)$. Then in \cite{klondon} the author calculated the class of all $\rho-$minimisers between annuli $\A(r,1)$ and $\A(\tau, \sigma)$. They are up to the rotation given by
\begin{equation}\label{w}w(se^{it}) =
p(s)e^{it}=q^{-1}(s)e^{it},\end{equation} where
\begin{equation}\label{var}q(y) =\exp\left(\int_{Q}^y
\frac{dy}{\sqrt{y^2+c\rho^{-1}(y) }}\right), \ q\le y\le Q,
\end{equation} and $c$ is a constant satisfying the condition:
\begin{equation}\label{unt} c\ge -y^2\rho(y), \;\text{for} \; q\le y\le Q.
\end{equation} Then $w$ is a $\rho$-harmonic mapping between annuli
$\A=\A(r,1)$ and $\A^*=\A^*(q,Q)$, where
\begin{equation}\label{rr}r=\exp\left(\int_{Q}^q \frac{dy}{\sqrt{y^2+c\rho^{-1}(y)
}}\right).\end{equation} The harmonic mapping $w$ is normalized by
$$w( e^{it})=Q e^{it}.$$ The mapping $w=h^c(z)$ is a diffeomorphism.
Further \begin{equation}\begin{split}|D w|&\bydef \max\{|Dw(z) h: |h|=1\}\\&=|\partial_z w|+|\partial_{\bar z} w|\\&=\frac{{\abs{p(s)+s p'(s)}}+\abs{p(s)-s p'(s)}}{2 s}\\&=\max\left\{\frac{p(s)}{s}, p'(s)\right\},\end{split}\end{equation}
and
\begin{equation}\begin{split}l(D w)&\bydef \min\{|Dw(z) h|: |h|=1\}\\&=|\partial_z w|-|\partial_{\bar z} w|\\&=\frac{\abs{p(s)+s p'(s)}-\abs{p(s)-s p'(s)}}{2 s}\\&=\min\left\{\frac{p(s)}{s}, p'(s)\right\}.\end{split}\end{equation}
By using \eqref{var} we get
\begin{equation}\label{pprim}p'(s)=\frac{1}{q'(p(s))}=\frac{\sqrt{p(s)^2+c\rho^{-2}(p(s)) }}{s}.\end{equation}
Thus \begin{equation}\label{ddd}\begin{split}|D w(se^{it})|&=\frac{\sqrt{p(s)^2+\max\{c,0\}\rho^{-1}(p(s)) }}{s}\\&\le \frac{\sqrt{Q^2+\max\{c,0\}M^2 }}{r}<\infty.\end{split}\end{equation}
This implies that $w$ is Lipschitz continuous on $\A(r,1)$. On the other hand we have that
$$l(D w)=\frac{\sqrt{p(s)^2+c\rho^{-1}(p(s)) }}{s}.$$  Thus for $c\ge 0$ we have

 \begin{equation}\label{qqq}
 l(D w)\ge q>0.
 \end{equation}
  Now \eqref{ddd} and \eqref{qqq} imply that $w$ is bi-Lipschitz in $\overline{\A(r,1)}$ for $c\ge 0$. However it is not bi-Lipschitz in general, i.e. for $c<0$.
  Indeed $l(D w)$ can be equal to zero for \begin{equation}\label{ccirc}c=c_\circ=-\min\{p^2(s)\rho(p(s)):s\in[r,1]\}.\end{equation}
  It should be noted that the condition $c\ge 0$, in view of \eqref{rr}  is equivalent with the condition

\begin{equation}\label{momo}\mathrm{Mod}(D)\le \mathrm{Mod}(\Omega).\end{equation}

 The minimiser $w$ is not bi-Lipschitz for the so-called critical J. C. C. Nitsche configuration of annuli: $\A(r_\circ,1)$ and $\A(q,Q)$, where
 \begin{equation}\label{rciri}r_\circ=\exp\left(\int_{Q}^q \frac{dy}{\sqrt{y^2+c_\circ\rho^{-1}(y)
}}\right).\end{equation}
 In particular for Eulcidean metric we have the following critical J. C. C. Nitsche configuration of annuli. For $0<r<1$ the mapping
\begin{equation}\label{eqp}w(z)=\frac{r^2+|z|^2}{\bar z(1+r^2)}\end{equation} is a harmonic minimiser (see \cite{AIM}) of the Euclidean energy of
mappings between $\mathbb{A}(r,1)$ and $\mathbb{A}(\frac{2r}{1+r^2},1)$, however $|w_z|=|w_{\bar z}|=\frac{1}{1+r^2}$ for $|z|=r$, and so $w$ is not
bi-Lipschitz. Those two annuli make the so-called critical configuration of annuli. Those configurations are important in framework of J. C. C.
Nitsche conjecture solved by Iwaniec, Kovalev and Onninen in \cite{IKO2} after some partial results given by Lyzzaik \cite{L}, Weitsman \cite{W} and Kalaj
\cite{isrg}.

The following conjecture is motivated by the previous observation.
\begin{conjecture}
 Assume that $D$ and $\Omega$ are doubly connected domains with smooth boundaries.
 Assume that $\rho$ is a smooth non-vanishing metric defined in the closure of $\Omega$. If $\mathrm{Mod}(D)\le\mathrm{Mod}(\Omega)$
 then the minimiser of $\rho-$energy is globally bi-Lipschitz continuous and has smooth extension up to the boundary.
\end{conjecture}

\end{document}